\documentclass{article}
\usepackage{xcolor}
\definecolor{grey}{rgb}{0.5,0.5,0.5}
\usepackage{amsmath}
\usepackage[left=3.0cm, right=3.0cm, bottom=2.7cm, top=2.7cm]{geometry}
\usepackage{hyperref}
\usepackage{cleveref}
\usepackage{cancel,xcolor}
\newcommand{\Cancel}[2][black]{{\color{#1}\cancel{\color{black}#2}}}
\usepackage{dsfont}
\usepackage{amssymb}
\usepackage{mathtools}
\usepackage{tikz-cd}
\usepackage{amsthm}
\usepackage{graphicx}
\usepackage{graphics}
\usepackage{mathrsfs}
\usepackage{cleveref}
\usepackage{thmtools}
\usepackage{enumitem}
\usepackage{bbm}
\usepackage{dsfont}
\usepackage[toc,page]{appendix}

\newlist{steps}{enumerate}{1}
\usepackage{mathtools}
\makeatletter
\newcommand{\xMapsto}[2][]{\ext@arrow 0599{\Mapstofill@}{#1}{#2}}
\def\Mapstofill@{\arrowfill@{\Mapstochar\Relbar}\Relbar\Rightarrow}
\makeatother
\usepackage{array}
\setlist[steps, 1]{label = Step \arabic*:}
\theoremstyle{plain}
\usepackage{footnote}
\makesavenoteenv{tabular}
\newtheorem*{theorem*}{Theorem}
\newtheorem{theorem}{Theorem}[subsection]
\newtheorem{definition}[theorem]{Definition}

\newtheorem{proposition}[theorem]{Proposition}
\newtheorem{remark}[theorem]{Remark}
\newtheorem{corollary}[theorem]{Corollary}
\newtheorem*{corollary*}{Corollary}
\newtheorem*{proposition*}{Proposition}
\newtheorem{definition*}{Definition}
\newtheorem{exmp}[theorem]{Example}

\newtheorem{lemma}[theorem]{Lemma}
\newtheorem*{lemma*}{Lemma}

\numberwithin{equation}{subsection}
\usepackage{rotating}
\usepackage{comment}
\usepackage{mathabx,epsfig}

\usepackage[nottoc,notlot,notlof]{tocbibind} 

\title{A Simple Description of the Hyperk\"{a}hler Structure of the Cotangent Bundle of Projective Space via Quantization}
\author{Joshua Lackman\footnote{Beijing International Center for Mathematical Research, Peking University. E-mail address: josh@pku.edu.cn}}
\date{}

\begin{document}

\maketitle
\begin{abstract}
\noindent 
Quantization identifies the cotangent bundle of projective space with the (non–Hermitian) rank–$1$ projections of a Hilbert space.\ We use this identification to study the natural geometric structures of these cotangent bundles and those of Grassmanians.\ In particular, we show that the quantization map is an isometric and complex embedding $\textup{T}^*\mathbb{P}\mathcal{H}\xhookrightarrow{}\mathcal{B}(\mathcal{H})\backslash \{0\}.$ Here, the metric on the domain is the hyperk\"{a}hler metric and the metric on the codomain is the one whose K\"{a}hler potential is the Hilbert–Schmidt norm.\ The K\"{a}hler potential pulled back to $\textup{T}^*\mathbb{P}\mathcal{H}$ equals the trace–class norm.\ Using this, we give a complete, simple and explicit description of the hyperk\"{a}hler structure.\ Our constructions are functorial, coordinate-free and reduction-free. 
\end{abstract}
\tableofcontents
\section{Introduction}
In Berezin's framework, a quantization of a real symplectic manifold is a \textit{supercomplete} embedding into the rank–$1$ Hermitian projections of a Hilbert space, where the latter is the space of quantum states and states in the embedded submanifold are the coherent states (\cite{ber1}, \cite{pol0}).\ The nicest examples are given by complex embeddings that are invariant under an irreducible representation — these are supercomplete due to Schur's lemma, which equivalently means that the constant function $1$ quantizes to the identity operator.\ The fundamental example for this framework is given by projective space $\mathbb{P}\mathcal{H},$ whose quantization identifies it exactly with the space of rank–$1$ Hermitian projections of $\mathcal{H}.$\footnote{This formulation is based on the observation that a quantization of \textit{all} observables immediately follows from a quantization of classical states, via a simple integral formula.\ By contrast, geometric quantization doesn't quantize any states and very few observables.} Quantization via almost complex structures provides more general examples (\cite{bord}, \cite{bor}), eg. Toeplitz quantization.
\\\\Extending Berezin's framework (\cite{Lackman1}), a quantization of a holomorphic symplectic manifold is a supercomplete embedding into the \textit{complexification} of the space of rank–$1$ Hermitian projections of a Hilbert space — this is the space of \textit{all} rank–$1$ projections, ie.\ it includes non–Hermitian ones. Here, the fundamental example is given by $\textup{T}^*\mathbb{P}\mathcal{H},$ whose quantization identifies it \textit{exactly} with the space of \textit{all} rank–1 projections of $\mathcal{H}.$ This is as an affine subvariety of $\mathcal{B}(\mathcal{H})$ (the bounded linear operators of $\mathcal{H}).$ 
\\\\Unlike $\mathbb{P}\mathcal{H},$ $\textup{T}^*\mathbb{P}\mathcal{H}$ carries many $U(\mathcal{H})$–invariant K\"{a}hler forms and holomorphic symplectic forms, and one can consider embeddings with respect to each of these.\ We record its hyperk\"{a}hler structure below.\ In the case of $\mathcal{H}=\mathbb{C}^2,$ it agrees with that of Eguchi–Hanson (\cite{biq}, \cite{calabi}). We note that, the complement of zero in any Hilbert space is a K\"{a}hler manifold, with K\"{a}hler potential given by the norm.
\begin{proposition}
Assume the embedding $\textup{T}^*\mathbb{P}\mathcal{H}\xhookrightarrow{}\mathcal{B}(\mathcal{H})\backslash\{0\}.$ We have a hyperk\"{a}hler structure given as follows:\ the Hilbert–Schmidt norm $\textup{T}^*\mathbb{P}\mathcal{H}\to\mathbb{R},\,q\mapsto\|q\|$ is a K\"{a}hler potential for the Riemannian metric given by the real part of
\begin{equation}
   \mathbf{h}(A,B)= \frac{1}{\|q\|}\textup{Tr}(A^*B)-\frac{1}{2\|q\|^3}\textup{Tr}(A^*q)\textup{Tr}(Bq^*)\;,
\end{equation}
where $A,B\in \textup{T}_q\textup{T}^*\mathbb{P}\mathcal{H}\subset \mathcal{B}(\mathcal{H})$ and the almost complex structure is given by $IA:=iA.$ The $I$–holomorphic symplectic form is given by
\begin{equation}\label{omega}
    \Omega(A,B)=i\textup{Tr}(q[A,B])
\end{equation}
and the equation $\mathbf{h}(\mathbf{J}A,B)=\Omega(A,B)$ defines an integrable almost complex structure $\mathbf{J}$ that anticommutes with $I.$ Explicitly, it given by
\begin{equation}
    \mathbf{J}A=\frac{i}{\|q\|}[q,A^*]+\frac{i}{2\|q\|^3}\textup{Tr}(A^*q)[q,q^*]\;.
    \end{equation}
Furthermore, 
there is a third integrable almost complex structure $J$ that \textbf{commutes} with $I,$ given by
\begin{equation}
    JA=i[A,q],
\end{equation}
and it is such that $\Omega(JA,JB)=\Omega(A,B).$ 
    \end{proposition}
Of course, with this metric $\textup{T}^*\mathbb{P}\mathcal{H}\xhookrightarrow{}\mathcal{B}(\mathcal{H})\backslash\{0\}$ is an isometric and complex embedding.\footnote{For the two–dimensional case, an isometric embedding of the Eguchi–Hanson metric into $\mathbb{R}^{11}$ is described in \cite{hanson}.}\ The former's K\"{a}hler potential is equal to the trace–class norm. This is natural from the physics perspective since states in quantum statistical mechanics have unit trace–class norm.\ The tensor fields $\Omega,\mathbf{J},J$ are also defined on $\mathcal{B}(\mathcal{H})\backslash\{0\},$ but only on certain submanifolds is $\Omega$ closed and do $\mathbf{J},J$ square to $-1.$ We will also describe a fourth integrable almost complex structure $\hat{J}$ that commutes with $J.$
\\\\
Due to the mutual compatibility of the different geometric structures, the fixed point sets of antiholomorphic involutions of $\textup{T}^*\mathbb{P}\mathcal{H}$ tend to be symplectic and come with natural polarizations.\ In particular, such fixed point sets have comparable quantizations and deformation quantizations, a phenomenon that is important in the brane quantization of \cite{brane}.\ The anticommuting complex structures tend to induce K\"{a}hler polarizations, whereas the commuting complex structures tend to also induce Lagrangian polarizations (\cite{Lackman1}).\ The latter is due to the fact that the product of commuting complex structures is an involution, whereas the product of anticommuting complex structures is another complex structure. 
\\\\For the simplest example, $\textup{T}^*\mathbb{P}^1$ is identified with rank–$1$ projections of $\mathbb{C}^2,$ ie.\ matrices of the form
\begin{equation}
    \begin{pmatrix}
        z & y \\
        x & 1-z
        \end{pmatrix}
\end{equation}
for which $xy-z+z^2=0.$
We have $I$–antiholomorphic involutions whose fixed point sets are the sphere, unit disk and cylinder, respectively given by
\begin{equation}
  (x,y,z)\mapsto (\bar{y},\bar{x},\bar{z})\,,\;\;(x,y,z)\mapsto (-\bar{y},-\bar{x},\bar{z})\,,\;\;(x,y,z)\mapsto (\bar{x},\bar{y},\bar{z}) \,.
\end{equation}
These are all symplectic submanifolds with respect to $\Omega.$ The sphere and unit disk inherit K\"{a}hler polarizations, while the cylinder inherits K\"{a}hler and Lagrangian polarizations.\ As these are all preserved by $I$–antiholomorphic involutions, they have isomorphic deformation quantizations.\footnote{Precisely, we are considering only those functions that extend to holomorphic functions on the cotangent bundle. These isomorphisms aren't Hermitian, ie.\ they don't preserve real–valued functions since a function that is real–valued on one submanifold doesn't need to be real–valued on another. See \cite{brane}.} 
\\\\The map 
\begin{equation}\label{affine}
    (x,y,z)\mapsto    (i(x-y),x+y,1-2z)
\end{equation} 
is an isomorphism onto $\{(x,y,z)\in\mathbb{C}^3:x^2+y^2+z^2=1\},$ where the almost complex structure $J$ that commutes with $I$ is given by the complexified cross product. That is, for a vector $(a,b,c)$ at $(x,y,z),$ 
\begin{equation}
    J(a,b,c)=(a,b,c)\times (x,y,z)\;.
    \end{equation}
\part{Differential Geometry of $\textup{T}^*\mathbf{G}\mathcal{H}$}
In part 1, we will first describe the quantization $\textup{T}^*\mathbf{G}_n\mathcal{H}\xhookrightarrow{}\mathcal{B}(\mathcal{H}),$ whose image is the space of (non–Hermitian) rank–$n$ projections.\footnote{This can equivalently be done using a prequantum holomorphic line bundle for $\Omega$ of \cref{holm}. See \cite{Lackman1}.} We will then study this space of projections, including its complex structures, symplectic structures, Riemannian metrics, vector bundle structure, compactification, Poisson geometry and path integrals.
\\\\In part 2 we will describe the hyperk\"{a}hler structure of $\textup{T}^*\mathbb{P}\mathcal{H}.$ Only \cref{affinee}, \cref{hols} are needed to understand it.
\\\\The goal is to describe the geometry of $\textup{T}^*\mathbf{G}\mathcal{H}$ via its identification with rank–$n$ projections, without referencing the commutative world.\ In fact, we will show that we can do this purely at the level of $C^*$–algebras, without referencing a representation.\ As such, this paper is coordinate–free and reduction–free. Except for when necessary, we will drop the subscript $n$ from the notation and leave the dimension of the Grassmanian implicit.
\begin{remark}
Via the complexified Pl\"{u}cker embedding, we can embed $\textup{T}^*\mathbf{G}\mathcal{H}$ into $\textup{T}^*\mathbb{P}(\wedge^{n} \mathcal{H}),$ where $\wedge^{n} \mathcal{H}$ is the exterior power and is equal to the Hilbert space of a system consisting of $n$ fermions.
\end{remark}
\section{$\textup{T}^*\mathbf{G}\mathcal{H}$ as an Affine Subvariety of $\mathcal{B}(\mathcal{H})$}\label{affinee}
We will follow the exposition of \cite{Lackman1}. Before describing the embedding in \cref{emb} we will describe an isomorphism $\textup{T}^*_V\mathbf{G}\mathcal{H}\cong \textup{Hom}(V^{\perp},V).$
\begin{lemma}\label{tan1}
Under the standard identification of $\mathbf{G}_n\mathcal{H}$ with rank–$n$ Hermitian projections, we have that $\textup{T}_q\mathbf{G}\mathcal{H}\cong \{A\in \mathcal{B}(\mathcal{H}): qA+Aq=A\,,\;A=A^*\},$ where $q$ is a rank–$n$ Hermitian projection.
\end{lemma}
\begin{proof}
This follows from the fact that modulo terms of order $\mathcal{O}(\varepsilon^2),$ $A\in \textup{T}_q\mathbf{G}\mathcal{H}$ if and only if 
\begin{equation}
    (q+\varepsilon A)^2=q+\varepsilon A\,,\;\;(q+\varepsilon A)^*=q+\varepsilon A\;.
\end{equation}
\end{proof}
As a result, we get the following standard identification:
\begin{proposition}\label{stand}
$\textup{T}_q\mathbf{G}\mathcal{H}\cong \textup{Hom}(V,V^{\perp}),$ where $V$ is the image of $q.$
\end{proposition}
\begin{proof}
The equation $qA+Aq=A$ implies that $qAq=0,$ therefore 
\begin{equation}
    A\vert_V(V)\subset V^{\perp}\;.
\end{equation}
As a result, we have a map
\begin{equation}\label{vvperp}
\textup{T}_q\mathbf{G}\mathcal{H}\to  \textup{Hom}(V,V^{\perp})\,,\;\;A\mapsto A\vert_V\;.
\end{equation}
Conversely, given $T\in \textup{Hom}(V,V^{\perp})$ and using the natural splitting $\mathcal{H}=V\oplus V^{\perp},$ we let 
\begin{equation}
A:=T\oplus T^{^*}\in\mathcal{B}(\mathcal{H})\;.
\end{equation}
From $T\vert_{V^\perp}=T^*\vert_{V}=0$ it follows that
\begin{equation}
    qA+Aq=A\;.
\end{equation}
Therefore, we have a map
\begin{equation}
    \textup{Hom}(V,V^{\perp})\to\textup{T}_q\mathbf{G}\mathcal{H}\;,\;\;T\mapsto T\oplus T^*\;,
\end{equation}
and this is the inverse of \cref{vvperp}.
\end{proof}
\begin{proposition}\label{perf}
$\textup{Hom}(V^\perp,V)\times\textup{Hom}(V,V^\perp)\mapsto\mathbb{C}\,,\,(f,g)\mapsto \textup{Re}\big(\textup{Tr}(fg)\big)$ is a perfect pairing.
\end{proposition}
\begin{proof}
This follows from the fact that $\textup{Tr}(ff^*)>0.$
\end{proof}
\begin{proposition}\label{perpstand}
Letting $V$ be the image of $q$ and assuming the identification $\textup{T}_q\mathbf{G}\mathcal{H}\cong \textup{Hom}(V,V^{\perp}),$
\begin{equation}
\textup{Hom}(V^\perp,V)\xmapsto{\alpha} \textup{T}^*_q\mathbf{G}\mathcal{H}\;,\;\;f\mapsto \alpha_f\,,\,\alpha_f(g)= \textup{Re}(\textup{Tr}(fg))\end{equation}
is a linear isomorphism.
\end{proposition}
\begin{proof}
This follows immediately from \cref{stand}, \cref{perf}.
\end{proof}
As a corollary, we get the desired identification of $\textup{T}^*\mathbf{G}_n\mathcal{H}$ with rank–$n$ projections, ie.
\begin{corollary}\label{emb}
$\textup{T}^*\mathbf{G}_n\mathcal{H}\cong\{q\in\mathcal{B}(\mathcal{H}): q^2=q, \textup{Tr}(q)=n\}.$
\end{corollary}
\begin{proof}
Due to \cref{perpstand}, we just need to show that
\begin{equation}
   \textup{Hom}(V^\perp,V)\cong\{\textup{Projections onto }V\}\;. 
\end{equation}
Let $q$ be any projection onto $V.$ We define $f\in \textup{Hom}(V^\perp,V)$ by
\begin{equation}\label{grass}
f=q\vert_{V^\perp}\;.
\end{equation}
Conversely, let $f\in \textup{Hom}(V^\perp,V)$ and let $\pi(q)$ be the orthogonal projection onto $V.$ We assume that $f$ is defined on all of $\mathcal{H}$ by extending it by $0$ on $V.$ We get a projection onto $V$ by defining
\begin{equation}
    q=\pi(q)+f\;.
\end{equation}
This is the inverse of \cref{grass}.
\end{proof}
A similar identification was described in \cite{leung}. This identification, together with the following one, will be assumed throughout this paper:
\begin{lemma}\label{tan}
Under the identification of \cref{emb},
\begin{equation}\label{tan}
 \textup{T}(\textup{T}^*\mathbf{G}\mathcal{H})\cong \{(q,A)\in \textup{T}^*\mathbf{G}\mathcal{H}\times \mathcal{B}(\mathcal{H}): qA+Aq=A\}\;.
\end{equation}
\end{lemma}
\begin{proof}
This follows from the same argument as in \cref{tan1}
\end{proof}
Since $\textup{T}^*\mathbf{G}\mathcal{H}$ is an affine subvariety of $\mathcal{B}(\mathcal{H}),$ it already inherits a K\"{a}hler structure given by the inner product. 
\subsection{Important Identities}
There are several important and simple–to–prove identities regarding $\textup{T}\textup{T}^*\mathbf{G}\mathcal{H}.$ We will record them here:
\begin{proposition}\label{identities}
Let $A,B\in\textup{T}_q\textup{T}^*\mathbf{G}\mathcal{H},$ ie.\  $qA+Aq=A,\,qB+Bq=B,$ and let $M\in\mathcal{B}(\mathcal{H}).$ Then
\begin{enumerate}
    \item $qAq=0.$
    \item $[q,[q,A]]=A.$
    \item\label{vec} $q[q,M]+[q,M]q=[q,M],$ ie.\ $[q,M]\in \textup{T}_q\textup{T}^*\mathbf{G}\mathcal{H},$
    \item $[q,AB]=0.$ In particular, $[q,[A,B]]=0.$
    \item\label{jajab} $i[A,q]i[B,q]=AB.$ In particular, $[i[A,q],i[B,q]]=[A,B].$
\end{enumerate}
\end{proposition}
\section{Differential Geometry of $\textup{T}^*\mathbf{G}\mathcal{H}$}
In this section, we'll discuss the basic differential geometry of 
$\textup{T}^*\mathbf{G}\mathcal{H}$ from the perspective of its quantization, ie.\ the space of idempotents. The point is to define all differential–geometric concepts at the level of $C^*$–algebras.\ It's not immediately obvious, but this can be done as simply as it is done on $\mathbb{R}^n.$ In particular, we will describe the first three of four (algebraically independent) integrable almost complex structures $I,J,\hat{J}, \mathbf{J}$ such that 
\begin{equation*}
IJ=JI\;,\;\;\hat{J}J=J\hat{J}\;,\,I\mathbf{J}=-\mathbf{J}I\;,
\end{equation*}
together with an $I$–holomorphic symplectic form for which $J$ is a pointwise symplectic map.\ We will also give an intrinsic description of the exterior derivative and vector bundle structure.
\\\\There are some basic results concerning the space of idempotents that are proved in \cite{Lackman1} that we will assume.
\subsection{Two Commuting Complex Structures, Symplectic Form and Exterior Derivative}\label{hols}
We will use \cref{identities} in this section.
\begin{lemma}
Let $q\in \textup{T}^*\mathbf{G}\mathcal{H}.$ Then $A\mapsto i[A,q]$ defines an endomorphism of  $\textup{T}_q(\textup{T}^*\mathbf{G}\mathcal{H})$ that squares to $-1.$
\end{lemma}
\begin{proof}
Since $qAq=0,$ it follows that
\begin{align}
 &\nonumber i[A,q]q+qi[A,q]=iAq-iqA=i[A,q]\;,
 \\& i[i[A,q],q]=-[Aq-qA,q]=-(Aq-qA)q+q(Aq-qA)=-(Aq+qA)=-A\;.
 \end{align}
 The first line shows that the map preserves $\textup{T}_q(\textup{T}^*\mathbf{G}\mathcal{H})$ and the second line shows that it squares to $-1.$
\end{proof}
\begin{definition}`(\cite{Lackman1})
Let $\mathcal{H}$ be a complex Hilbert space.\ We define (integrable) \textbf{commuting} almost complex structures $I, J$ on $\textup{T}^*\mathbf{G}\mathcal{H}$ by 
\begin{equation}
    A\xmapsto{I} iA\;,\;\; A\xmapsto{J} i[A,q]
    \end{equation}
for $A\in \textup{T}_q\textup{T}^*\mathbf{G}\mathcal{H}.$ Furthermore, we define an involution of $\textup{T}\textup{T}^*\mathbf{G}\mathcal{H}$ the $K:=IJ.$ 
\end{definition}
\begin{definition}\label{holm}(\cite{Lackman1})
We define an $I$–holomorphic symplectic form by
\begin{equation}
    \Omega(A,B)=i\textup{Tr}(q[A,B])\;,
\end{equation}
where $A,B\in\textup{T}_q\textup{T}^*\mathbf{G}\mathcal{H}.$
\end{definition}
In \cite{Lackman1}, we showed that $\Omega$ is closed by showing that it is the trace of the curvature of a connection on a vector bundle. Here, we will explain how to directly and intrinsically compute its exterior derivative.\ To do so, we use the fact that
$\textup{T}^*\mathbf{G}\mathcal{H}$ has natural vector fields, which allows us to give \textit{pointwise} definitions of tensor fields that are normally defined using vector fields or coordinates.
\begin{lemma}\label{derivative}
Let $\omega$ be an $n$–form on $\textup{T}^*\mathbf{G}\mathcal{H}$ and for $A\in\textup{T}_q\textup{T}^*\mathbf{G}\mathcal{H}$ consider the vector field 
\begin{equation}
   p\mapsto \mathbf{A}_p:=[p,[q,A]]\;.\footnote{This is a vector field by identity \ref{vec}.}
\end{equation}
Then for $A_0,\ldots,A_{n}\in\textup{T}_q\textup{T}^*\mathbf{G}\mathcal{H},$ 
\begin{equation}
d\omega(A_0,\ldots,A_{n})=\sum_{j=0}^{n}(-1)^jA_{j}\omega(\mathbf{A}_{0},\ldots,\hat{\mathbf{A}}_{j},\ldots,\mathbf{A}_{n})\;.
\end{equation}
\end{lemma}
\begin{proof}
This follows from the Lie bracket formula for the exterior derivative together with the facts that $\mathbf{A}_q=A$ and for $A,B\in\textup{T}_q\textup{T}^*\mathbf{G}\mathcal{H},$ $[\mathbf{A},\mathbf{B}]\vert_{q}=0.$
\end{proof}
The following lemma is also useful:
\begin{lemma}\label{odd}
For any $k\ge 0,$ let $A_1,\ldots,A_{2k+1}\in \textup{T}_q\textup{T}^*\mathbf{G}\mathcal{H}.$ Then 
\begin{equation}
\textup{Tr}(A_1\cdots A_{2k+1})=0\;.
\end{equation}
\end{lemma}
\begin{proof}
Due to the fact that $A=qA+Aq,$ and the invariance of the trace under cyclic permutations, we only need to prove that this is true for $A_1$ such that $A_1=qA_1.$\footnote{$A\in \textup{T}_q\textup{T}^*\mathbf{G}\mathcal{H}$ implies that $qA,Aq\in \textup{T}_q\textup{T}^*\mathbf{G}\mathcal{H}.$} For this, we use that $q^2=q$ and that for $A,B\in\textup{T}_q\textup{T}^*\mathbf{G}\mathcal{H},\, [q,AB]=0,$ from which it follows that
\begin{equation}
\textup{Tr}(qA_1\cdots A_{2k+1})=\textup{Tr}(qA_1\cdots A_{2k+1}q)=\textup{Tr}(qA_1q\cdots A_{2k+1})=0\;,
\end{equation}
where the final equality follows from the fact that $qA_1q=0.$
\end{proof}
\begin{corollary}
$d\Omega=0$ and $\Omega(JA,JB)=\Omega(A,B).$
\end{corollary}
\begin{proof}
The first part follows from the previous two lemmas, since
\begin{equation}
A\Omega(\mathbf{B},\mathbf{C})=i\textup{Tr}(A[B,C])+i\textup{Tr}(q[[A,[q,B]],C])+i\textup{Tr}(q[B,[A,[q,C]]])\;,
\end{equation}
and the latter lemma shows that each of these terms are zero. The second part follows from identity \ref{jajab}.
\end{proof}
\begin{remark}
The exterior derivative can be computed slightly more easily by observing that $\Omega$ has a natural extension to $\mathcal{B}(\mathcal{H}).$ However, the method we've described is intrinsic to the $C^*$–algebra.
\end{remark}
\subsection{Vector Bundle Structure}
Here we will describe the vector bundle structure of $\textup{T}^*\mathbf{G}\mathcal{H}$ from the perspective of its quantization. We recall that the zero section is identified with the space of Hermitian projections.
\subsubsection{Projection Map}
We'll first define the projection map and discuss its compatibility with $I, J.$ We'll then describe its derivative.
\begin{definition}
We define a map $\pi:\textup{T}^*\mathbf{G}\mathcal{H}\to \mathbf{G}\mathcal{H},$ where $\pi(q)$ is the orthogonal projection onto the image of $q.$    
\end{definition}
\begin{proposition}
Two projections $q,q'\in \textup{T}^*\mathbf{G}\mathcal{H}$ have the same image under $\pi$ if and only if $qq'=q'.$ 
\end{proposition}
\begin{proof}
If $qq'=q'$ then $q$ fixes the image of $q'.$ Since by assumption $q,q'$ have the same rank, their images must therefore be equal.
\end{proof}
\begin{corollary}
The distributions given by the $\pm 1$–eigenbundles of $K$ are integrable.\ The \\$+1$–eigenbundle is given by vectors tangent to the leaves of $\pi$ and the the $-1$–eigenbundle is given by vectors tangent to the leaves of $q\mapsto \pi(q^*).$
\end{corollary}
\begin{proposition}\label{anti}
Let $\mathcal{H}$ be $n$–dimensional. The projection map $\pi$ is $I$ and $J$–holomorphic.\ The adjoint map $^*$ is $I$–antiholomorphic and $J$–holomorphic. The map
\begin{equation}
\textup{T}^*\mathbf{G}_k\mathcal{H}\to \textup{T}^*\mathbf{G}_{n-k}\mathcal{H}\;,\;\;q\mapsto 1-q  
\end{equation}
is $I$–holomorphic, $J$–antiholomorphic and commutes with $^*.$ In the case that $n=2k,$ the composition 
\begin{equation}
\textup{T}^*\mathbf{G}_k\mathcal{H}\to \textup{T}^*\mathbf{G}_{k}\mathcal{H}\;,\;\;q\mapsto 1-q^*  
\end{equation}
is an $I$ and $J$–antiholomorphic involution and it is also a vector bundle map over $1-q.$
\end{proposition}
This next lemma describes $\pi_*:\textup{T}\textup{T}^*\mathbf{G}\mathcal{H}\to \textup{T}\mathbf{G}\mathcal{H}$
\begin{lemma}\label{pistar}
 Let $A\in\textup{T}_q\textup{T}^*\mathbf{G}\mathcal{H}.$ Then $\pi_*(A)\in\textup{T}_{\pi(q)}\mathbf{G}\mathcal{H}$ is the unique vector such that
\begin{equation}
    q\pi_*(A)+A\pi(q)=\pi_*(A)\;.
\end{equation}
\end{lemma}
\begin{proof}
Modulo terms of order $\mathcal{O}(\varepsilon^2),$ $\pi_*(A)$ is defined by the condition that 
\begin{equation}
    (q+\varepsilon A)(\pi(q)+\varepsilon \pi_*(A))=\pi(q)+\varepsilon \pi_*(A)\;.
    \end{equation}
Expanding proves the result.
\end{proof}
Recall that for any projection $q$ and $M\in\mathcal{B}(\mathcal{H}),$ $[q,M]\in \textup{T}_q\textup{T}^*\mathbf{G}\mathcal{H}$ (identity \ref{vec}). We have the following simple description of $\pi_*$:
\begin{corollary}\label{push}
If $M\in\mathcal{B}(\mathcal{H})$ is skew-hermitian, we have that
\begin{equation}
    \pi_*([q,M])=[\pi(q),M]\;.\end{equation}
If $M$ is Hermitian, we have that
\begin{equation}
    \pi_*([q,M])=[[\pi(q),M],\pi(q)]\;.
\end{equation}
\end{corollary}
We end this section with a simple formula for $\pi,$ in the case of projective space:
\begin{proposition}
$\pi:\textup{T}^*\mathbb{P}\mathcal{H}\to \mathbb{P}\mathcal{H}$ is given by 
\begin{equation}
\pi(q)=\frac{qq^*}{\textup{Tr}(qq^*)}\;.
\end{equation} 
\end{proposition}
\begin{proof}
This follows from the fact that for any rank–$1$ projection $q$ and endomorphism $M,$ $qMq=\textup{Tr}(qM)q.$
\end{proof}
\subsubsection{Vector Space Structure of Fibers and Compactification}
This next lemma shows that the leaves of $\pi$ are naturally affine spaces. This will be used to describe the vector space structures.
\begin{lemma}
Let $\mathcal{L}$ be a leaf containing $\pi(q).$ The map
\begin{equation}
    \mathcal{L}\times T_{\pi(q)}\mathcal{L}\to\mathcal{L}\;,\;\;(q,A)\mapsto q+A
\end{equation}
is well–defined and describes a free and transitive action of $T_{\pi(q)}\mathcal{L}$ on $ \mathcal{L}.$
\end{lemma}
\begin{proof}
Since $A$ is tangent to a leaf, $qA=A.$ Therefore, $A^2=Aq=0,$ hence $(q+A)^2=q+A$ and $q(q+A)=q+A.$ 
\end{proof}
As a result of the previous lemma, we can describe the vector bundle structure as follows:
\begin{definition}
Let $q,q'$ be in the same leaf and let $r\in\mathbb{C}.$ We define
\begin{align}
    & r\cdot q:= (1-r)\pi(q)+rq\;,
    \\& q\oplus q':=q+q'-\pi(q)\;.
\end{align}
This equips each leaf with the structure of a vector space.
\end{definition}
This next proposition concerns the Fubini–Study metric–induced vector bundle isomorphism $\textup{T}{G}\mathcal{H}\to \textup{T}^*{G}\mathcal{H}$
\begin{proposition}
With respect to the identification \cref{perpstand}, the map
\begin{equation}
    \textup{T}_q{G}\mathcal{H}\to \textup{T}^*_q{G}\mathcal{H}\;,\;\;A\mapsto \big(B\mapsto \textup{Tr}(AB)\big)
\end{equation}
is given by $A\mapsto qA,$ ie.
\begin{equation}
    \textup{Tr}(AB)=\textup{Re}(\textup{Tr}(qAB)).
\end{equation}
\end{proposition}
\begin{proof}
$\textup{Re}(\textup{Tr}(qAB))=\textup{Tr}(qAB)+\textup{Tr}(BAq)=\textup{Tr}(qAB)+\textup{Tr}(AqB)=\textup{Tr}(AB).$
\end{proof}
\begin{proposition}
We have an embedding
\begin{equation}\label{compa}
    \textup{T}\mathbb{P}\mathcal{H}\xhookrightarrow{f} \mathbb{P}\mathcal{H}\times \mathbb{P}\mathcal{H}\;,\;\;(q,A)\mapsto\Big(q,\frac{q+A+AqA}{1+\textup{Tr}(qA^2)}\Big)
\end{equation}
whose image consists exactly of all $(p,q)$ such that $pq\ne 0.$ In addition, $t\mapsto f(q,tA)$ is a curve in $\mathbb{P}\mathcal{H}$ beginning at $q$ and whose derivative at $t=0$ is $A.$ Furthermore, it converges to 
\begin{equation}\label{limit}
    \frac{AqA}{\textup{Tr}(qA^2)}
\end{equation}
as $t\to\infty.$ This limiting point is invariant under nonzero complex multiplication of $A,$ ie.\ for 
\begin{equation}
    A'=(x+yJ)A\;,\;x^2+y^2\ne 0\;,
\end{equation}
we have that
\begin{equation}
    \frac{A'qA'}{\textup{Tr}(qA'^2)}=\frac{AqA}{\textup{Tr}(qA^2)}\;.
\end{equation}Furthermore, $q$ and this limiting point multiply to zero.
\end{proposition}
This map describes a compactification of $\textup{T}\mathbb{P}\mathcal{H},$ an idea we'll return to in \cref{comp}, where we describe the compactification of $\textup{T}^*\mathbf{G}\mathcal{H}$ and where more detail will be provided. This is equivalent to the standard vector bundle compactification in algebraic geometry,\footnote{This is given by taking the direct sum with the trivial bundle and projectivizing the fibers of the resulting bundle.} ie.\ it is a fiberwise compactification of complex vector spaces.
\begin{remark}\label{intr}
The canonical curves $t\mapsto f(q,tA)$ allow us to compute exterior derivatives intrinsically. Essentially, this means that $\mathbb{P}\mathcal{H}$ can be used as a local model for manifolds, rather than Euclidean space. The van Est map can be used to give an intrinsic definition of integration, see \cite{Lackman2}.
    \end{remark}

\subsubsection{The Connection and Third Commuting Complex Structure}
\begin{definition}\label{connection}
We define a connection on $\pi:\textup{T}^*\mathbf{G}\mathcal{H}\to \mathbf{G}\mathcal{H},$ given by the splitting  
\begin{equation}
\textup{T}_{\pi(q)}\mathbf{G}\mathcal{H}\xrightarrow[]{H}\textup{T}_q \textup{T}^*\mathbf{G}\mathcal{H}\;,\;\;   H(A)=[q,[\pi(q),A]]\;.
\end{equation}
\end{definition}
\begin{lemma}\label{part2}
The curvature of the connection of \cref{connection} is given by
\begin{equation}
    \mathcal{F}_q(A,B)C=C[A,B]-[A,B]C\;.
\end{equation}
\end{lemma}
\begin{proof}
This is a straightforward computation using \cref{derivative}.
\end{proof}
\begin{proposition}
The $(0,2)$–part of $\mathcal{F}$ vanishes, ie. for all $q\in\mathbf{G}\mathcal{H},$ $\mathcal{F}_q(A+iJA,B+iJB)=0.$
\end{proposition}
\begin{proof}
Direct computation.
\end{proof}
\begin{definition}
We let $\hat{J}$ be the almost complex structure determined by \cref{connection}.
\end{definition}
In fact, $\hat{J}$ is the natural complex structure of $\textup{T}^*\mathbf{G}\mathcal{H}$ by considering it as the cotangent bundle of a complex manifold.
\begin{proposition}\label{idhat}
Let $A\in\textup{T}_q\mathbf{G}\mathcal{H}.$ Then $\pi_*((HJA)^*)=i[A,\pi(q^*)].$
\end{proposition}
\begin{proposition}
$[J,\hat{J}]=0.$
\end{proposition}
\begin{proof}
This follows from the formula for $\hat{J}$ given in \cref{ext}.
 \end{proof}
\section{Compactification of $\textup{T}^*\mathbf{G}\mathcal{H}$ and Another Description of $I,J,\hat{J}$}\label{comp}
Here we will describe the compactification $\textup{T}^*\mathbf{G}\mathcal{H}\xhookrightarrow{}\mathbf{G}\mathcal{H}\times\mathbf{G}\mathcal{H}$ and show that $I,J,\hat{J}$ extend to $\mathbf{G}\mathcal{H}\times\mathbf{G}\mathcal{H}.$ We obtain simple formulas.
\begin{lemma}
Suppose $p,q\in\textup{T}^*\mathbf{G}\mathcal{H}$ are such that $\pi(p)=\pi(q),\,\pi(p^*)=\pi(q^*).$ Then $p=q.$
\begin{proof}
From the assumptions it follows that $q^*p^*=p^*.$ Taking adjoints gives $pq=p.$ Since the assumptions also imply that $pq=q,$ the result follows.
\end{proof}
\end{lemma}
\begin{corollary}
With respect to the product complex structure on the codomain, the map \begin{equation}\label{com}
\textup{T}^*\mathbf{G}\mathcal{H}\xhookrightarrow{}\mathbf{G}\mathcal{H}\times\mathbf{G}\mathcal{H}\;,\;\;q\mapsto (\pi(q),\pi(q^*))  
\end{equation} 
is a $J$-holomorphic, open and dense embedding.\ Its image exactly consists of all $(p,q)$ such that 
\begin{equation}
\mathcal{H}=p(\mathcal{H})\oplus q(\mathcal{H})^{\perp}\;.
\end{equation}
With respect to the complex structure $(J,-J),$\footnote{$\mathbf{G}\mathcal{H}$ is a $J$–complex submanifold of its cotangent bundle, so we use the same notation for its complex structure.} it is $I$–holomorphic.
\end{corollary}
\begin{proof}
That its image is open and dense follows from the description of its image.\ The statement about it being $J$ (respecively, $I$) holomorphic follows from previous results. Suppose that $\mathcal{H}=p(\mathcal{H})\oplus q(\mathcal{H})^{\perp}.$ Define
\begin{equation}
A\vert_{p(\mathcal{H})}:p(\mathcal{H})\to q(\mathcal{H})^{\perp}\;,\;\;A=1-q
\end{equation}
and extend it by zero on $q(\mathcal{H})^\perp.$ Then
\begin{equation}
Aq=A\;,\;\;qA=0\;,
\end{equation}
and therefore $qA^*=A^*,\,A^*q=0,$ which implies that $\pi(q+A^*)=q.$ Furthermore, $(q+A)p=p,$
which proves that $\pi(q+A)=p.$ Conversely, suppose that $q'\mapsto (p,q).$ Then for $0\ne v\in p(\mathcal{H}),$
\begin{equation}
    0<\langle v,v\rangle =\langle q'^*v,v\rangle=\langle v,q'v\rangle\;,
\end{equation}
which shows that $v\not\in q(\mathcal{H})^{\perp}.$ This implies the result.
\end{proof}
\begin{corollary}\label{ext}
The complex structures $I,J, \hat{J},$ the involution $^*$ and the two projection maps $\textup{T}^*\mathbf{G}\mathcal{H}\to \mathbf{G}\mathcal{H}$ all extend to $\mathbf{G}\mathcal{H}\times \mathbf{G}\mathcal{H}.$ For $(A,B)\in \textup{T}_{(p,q)}(\mathbf{G}\mathcal{H}\times\mathbf{G}\mathcal{H}),$ the extension of $\hat{J}$ is given by
\begin{equation}
(A,B)\mapsto (J_pA,-J_qB+J_qA-J_q^2J_pA)\;,
\end{equation}
where $J_xA=i[A,x].$
\end{corollary}
\begin{corollary}
The image of
\begin{equation}
\textup{T}^*\mathbb{P}\mathcal{H}\xhookrightarrow{}\mathbb{P}\mathcal{H}\times\mathbb{P}\mathcal{H}\;,\;\;q\mapsto (\pi(q),\pi(q^*))
\end{equation}
exactly consists of all $(p,q)$ such that $pq\ne 0.$ Furthermore, its complement is preserved by $(J,-J).$
\begin{remark}
The set $\{pq=0\}$ can be described as the zero set of a section of a line bundle (the box product of the anticanonical and canonical bundles), see \cref{sect}. 
\end{remark}
\begin{proof}
The first statement follows from the previously corollary. For the second part, if $pq=0,$ then $(A,B)\in\textup{T}_{(p,q)}(\mathbb{P}\mathcal{H}\times\mathbb{P}\mathcal{H})$ if and only if $A\in \textup{T}_{p}\mathbb{P}\mathcal{H},\,B\in \textup{T}_{q}\mathbb{P}\mathcal{H}$ and
\begin{equation}
pB+Aq=0\;.
\end{equation}
Note that, this equation implies that $pB=pBq,\,Aq=pAq.$ We then have that
\begin{equation}
-p[B,q]+[A,p]q=-pBq-pAq=-(pB+Aq)=0\;,
\end{equation}
which implies the result.
\end{proof}
\end{corollary}
\begin{remark}
\Cref{com} identifies a complex manifold, $\textup{T}^*\mathbf{G}\mathcal{H},$ with a submanifold of $\mathbf{G}\mathcal{H}\times\mathbf{G}\mathcal{H},$ which is the square of the fixed point set of an antiholomorphic involution. Loosely speaking, this is a decomposition into real and imaginary parts.
\end{remark}
\begin{exmp}
For $\textup{dim}\,\mathcal{H}=2,$ the complement of $\textup{T}^*\mathbb{P}\mathcal{H}\subset \mathbb{P}\mathcal{H}\times \mathbb{P}\mathcal{H}$ is given by the image of $\mathbb{P}\mathcal{H}\xhookrightarrow{} \mathbb{P}\mathcal{H}\times \mathbb{P}\mathcal{H},\,q\mapsto (q,1-q).$
\end{exmp}
\section{Poisson Geometry}
Here, we'll describe the canonical Poisson bracket on the cotangent bundle and prove that there is a morphism of Lie algebras  
\begin{equation}
    (\mathcal{B}(\mathcal{H}),[\cdot,\cdot])\to (\textup{T}^*\mathbf{G}\mathcal{H},\{\cdot,\cdot\})\;.
    \end{equation}
There is also such a morphism of Lie algebras with $\{\cdot,\cdot\}$ given by the Poisson bracket of $\Omega,$ see (\cite{Lackman1}). This is in stark contrast to the well–known fact that there is no reasonable morphism in the opposite direction.
\begin{proposition}
The $I$-complexification of the tautological 1–form is given by
\begin{equation}
    (q,A)\mapsto -\textup{Tr}(\pi(q)A)\;,
\end{equation}
ie.\ the real part is the tautological 1–form.
\end{proposition}
\begin{proof}
Using \cref{pistar} and taking traces, we find that
\begin{equation}
    \textup{Tr}(q\pi_*(A))=-\textup{Tr}(A\pi(q))\;.
\end{equation}
The real part of the left side is the tautological 1–form and this proves the result.
\end{proof}
\begin{definition}
Associated to any skew-hermitian endomorhism $M$ of $\mathcal{H}$ is a vector field $q\mapsto[q,M]$ on $\mathbf{G}\mathcal{H}.$ By decomposing an endomorphism into its skew-hermitian parts as 
\begin{equation}
M=\frac{M-M^*}{2}+i\frac{M+M^*}{2i}\;,
\end{equation}
we get a map
\begin{equation}
    \mathcal{B}(\mathcal{H})\to \textup{T}_{\mathbb{C}}\mathbf{G}\mathcal{H}\;,M\mapsto \mathcal{X}_M\;,\;\;\mathcal{X}_M(q)=\Big[q,\frac{M-M^*}{2}\Big]+i\Big[q,\frac{M+M^*}{2i}\Big]\;.
\end{equation}
By pairing with the trace, these define fiberwise linear functions on $\textup{T}^*\mathbf{G}\mathcal{H},$ ie.
\begin{equation}
  \mathcal{B}(\mathcal{H})\xrightarrow[]{\widehat{}} C^{\omega}(\textup{T}^*\mathbf{G}\mathcal{H})\;,\;\;\hat{M}(q):= \textup{Tr}([\pi(q),q]M)\;.
\end{equation}
\end{definition}
Note that, $[\pi(q),q]=q-\pi(q)$ is the Euler–vector field.
\begin{proposition}\label{poiss}
$M\mapsto \mathcal{X}_M$ is a morphism of Lie algebras, ie.
\begin{equation}
   [\mathcal{X}_M,\mathcal{X}_N]=\mathcal{X}_{[M,N]}\;.
\end{equation}
Furthermore, with respect to the canonical Poisson structure on the cotangent bundle, $\,\widehat{}\,$ is a morphism of Lie algebras, ie.
 \begin{equation}
     \{\hat{M},\hat{N}\}=\widehat{[M,N]}\;.
 \end{equation}
\end{proposition}
\begin{proof}
The first part is a direct computation. The second part follows from the first part and the fact that the canonical Poisson bracket of functions on the cotangent bundle defined by vector fields on the base is the Lie bracket.
\end{proof}
\begin{definition}
Using the Hilbert–Schmidt inner product on $\mathcal{B}(\mathcal{H}),$ we define a noncommutative product $\star$ on the image of $\;\widehat{}\;,$ given by
\begin{equation}
    \hat{M}\star\hat{N}:=\widehat{MN}\;,
\end{equation}
where $M,N\in\textup{Ker}(\,\;\widehat{}\,\;)^{\perp}.$
\end{definition}
\begin{proposition}
$\hat{M}\star\hat{N}-\hat{N}\star\hat{M}=\{\hat{M},\hat{N}\}.$
\end{proposition}
\begin{proof}
This follows from \cref{poiss}.
\end{proof}
\section{The Idempotent Section (Path Integral)}
In \cite{Lackman1}, we considered path integrals of the form
    \begin{equation}
\int_{\gamma(0)=x}^{\gamma(1)=y}\mathcal{P}(\gamma)\,\mathcal{D}\gamma\;,
\end{equation}
where the domain of integration consists of all paths $\gamma$ between two points $x,y$ in the zero section of a vector bundle $\mathcal{E}\to M,$ and whose integrand denotes parallel transport between the corresponding fibers.\ 
This is 
a higher–rank generalization of the coherent state path integral, which is a version of Feynman's path integral and determines the Hilbert space in quantum theory — this is rigorously equivalent to Berezin's formulation of quantization (\cite{Lackman2}, \cite{pol0}). We mathematically formulated such a path integral as a normalized, Hermitian idempotent section of
\begin{equation}
\mathcal{E}^*\boxtimes\mathcal{E}\to M\times M
\end{equation}
whose covariant derivative at the diagonal is zero.
In practice, such sections tend to be the integral kernel of the orthogonal projection onto the space of holomorphic sections, eg. the Bergman kernel (\cite{ma}). We'll consider such a section here. As a first application, we'll find that in a neighborhood of each point in $\mathbf{G}\mathcal{H},$ $\textup{T}^*\mathbf{G}\mathcal{H}$ is canonically locally trivial. 
\\\\In the following, we make use of the identification of $\pi^{-1}(q)$ with all $A\in\mathcal{B}(\mathcal{H})$ such that $qA=A, Aq=0,$ ie.\ the identification is given by $A\mapsto q+A.$ 
\begin{definition}
We define a section $\mathcal{S}$ of $(\textup{T}^*\mathbf{G}\mathcal{H})^*\boxtimes \textup{T}^*\mathbf{G}\mathcal{H}\to \mathbf{G}\mathcal{H}\times \mathbf{G}\mathcal{H}\;,$ given by
\begin{equation}
A\xmapsto{\mathcal{S}(q,p)} pA-pAp\;. 
\end{equation}
Here, $(\textup{T}^*\mathbf{G}\mathcal{H})^*\boxtimes \textup{T}^*\mathbf{G}\mathcal{H}:= \pi_1^*(\textup{T}^*\mathbf{G}\mathcal{H})^*\otimes \pi_2^*\textup{T}^*\mathbf{G}\mathcal{H},$ where $\pi_1,\pi_2:\mathbf{G}\mathcal{H}\times \mathbf{G}\mathcal{H}\to  \mathbf{G}\mathcal{H}$ are the projections onto the first and second factor, respectively.
\end{definition}
To be clear, we are identifying $\mathcal{S}(q,p)$ with a morphism of vector spaces $\textup{T}^*_q\mathbf{G}\mathcal{H}\to \textup{T}^*_p\mathbf{G}\mathcal{H}.$
\begin{proposition}
For all $q\in\mathbf{G}\mathcal{H},$ $\mathcal{S}(q,q)$ is the identity and $\mathcal{S}$ is Hermitian with respect to the natural Hermitian metric on the fibers of $\pi,$ ie. for $A\in\textup{T}^*_q\mathbf{G}\mathcal{H},\,B\in\textup{T}^*_p\mathbf{G}\mathcal{H},$
\begin{equation}
    \textup{Tr}((A\mathcal{S}(q,p))^*B)=\textup{Tr}(A^*(B\mathcal{S}(p,q)))
\end{equation}
\end{proposition}
\begin{proof}
The first part follows from the fact that for $A\in\textup{T}_q^*\mathbf{G}\mathcal{H},$ $qA=A,\, qAq=0.$ The second part follows from
\begin{equation}
    \textup{Tr}((pA-pAp)^*B)=\textup{Tr}(A^*B)=\textup{Tr}(A^*(qB-qBq))\;.
\end{equation}
\end{proof}
\begin{proposition}
For each $q\in \mathbf{G}\mathcal{H},$ the set of $p\in \mathbf{G}\mathcal{H}$ for which
\begin{equation}\label{secti}
\textup{T}^*_q\mathbf{G}\mathcal{H}\to    \textup{T}^*_p\mathbf{G}\mathcal{H}\;,\;\;A\mapsto A\mathcal{S}(q,p)
\end{equation}
is an isomorphism is an open set containing $q.$ Furthermore, for each fixed $q$ and $A$ the section of $\textup{T}^*\mathbf{G}\mathcal{H}\to\mathbf{G}\mathcal{H}$ defined by \cref{secti} is $J$–holomorphic.
\end{proposition}
\begin{proof}
The first part follows from the fact that $\mathcal{S}$ is the identity on the diagonal, hence is an isomorphism there, and the fact that the rank of a vector bundle morphism is lower semi–continuous. That \cref{secti} defines a $J$–holomorphic section is a direct computation.
\end{proof}
\begin{proposition}\label{sect}
In the case of $\mathbb{P}\mathcal{H},$ \cref{secti} is an isomorphism for all $p$ such that $pq\ne 0$ and is zero otherwise.
\end{proposition}
\begin{proof}
We just need to prove that it is injective there. In this case, $pAp=\textup{Tr}(pA).$ Suppose that 
\begin{equation}\label{one}
    pA-\textup{Tr}(pA)p=0\;.
\end{equation}
Since $Aq=0, pq\ne 0,$ it follows that 
\begin{equation}
\textup{Tr}(pA)=0\;,
\end{equation}
and therefore \cref{one} implies that $pA=0.$ Since the image of $q$ is 1–dimensional and $qA=A,$ $A$ is either zero or it surjects onto the image of $q.$ Therefore, since $pq\ne 0,$ $pA=0$ implies that $A=0.$
\end{proof}
In addition to giving local trivializations, such sections are useful for defining connections. 
\begin{definition}
Differentiating $\mathcal{S}$ in the second component at the diagonal determines a splitting of
\begin{equation}
\textup{T}\textup{T}^*\mathbf{G}\mathcal{H}\to \pi^*\textup{T}\mathbf{G}\mathcal{H}
\end{equation}   
which we denote by $\nabla_{\mathcal{S}}.$
\end{definition}
The following shows that the covariant derivative of $\mathcal{S}$ at the diagonal is zero: 
\begin{proposition}
The morphism $\textup{T}_{q}\mathbf{G}\mathcal{H}\to \textup{T}_{q+A}\textup{T}^*\mathbf{G}\mathcal{H}$
is given by 
\begin{equation}
    \nabla_{\mathcal{S}}B=B+[B,A]\;.
\end{equation}
This splitting agrees with \cref{connection}
\end{proposition}
\begin{proof}
This is a straightforward computation.
\end{proof}
\begin{proposition}
$\mathcal{S}$ is an idempotent section, ie.\ with respect to the volume form $dq$ of the Fubini–Study metric (after normalizing by a constant), for all $q,p\in \mathbf{G}\mathcal{H}$
\begin{equation}\label{con}
S(q,p)=\int_{\mathbf{G}\mathcal{H}}S(q,q')S(q',p)\,dq'\;.
\end{equation}
\end{proposition}
\begin{remark}
This proof is a variant of the standard proof that projective space is supercomplete (or overcomplete), using Schur's lemma, ie.\ a linear operator commuting with an irreducible representation is a multiple of the identity. See \cite{klauder4}.
\end{remark}
\begin{proof}
We have that
   \begin{equation}
       AS(q,q')S(q',p)=p\big(q'A-q'Aq'\big)-p\big(q'A-q'Aq'\big)p\;,
       \end{equation}
so the result follows if there exists a constant $\lambda \ne 0$ such that for all traceless $M,$
\begin{equation}\label{tra}
  \lambda M=\int_{\mathbf{G}\mathcal{H}}(qM-qMq) \,dq\;.
\end{equation}
Note that, $\mathcal{S}$ is an integral kernel and defines a trace–class operator on square–integrable sections of $\textup{T}^*\mathbf{G}\mathcal{H}\to \mathbf{G}\mathcal{H}.$ It is a continuous section that is the identity on the diagonal, and therefore the trace of this operator is the volume of $\mathbf{G}\mathcal{H}$ with respect to $dq$ (\cite{bris}, theorem 3.1). In particular, the trace is positive and therefore this operator can't be nilpotent.\ Therefore, the right of \cref{con} can't be zero for all $q,p.$ \Cref{tra} now follows from Schur's lemma: let $U\in\mathcal{B}(\mathcal{H})$ be unitary. Then by invarance of $dq$ under the unitary group,
\begin{equation}
U\int_{\mathbf{G}\mathcal{H}}q \,dq=U\int_{\mathbf{G}\mathcal{H}}U^*qU \,dq=\int_{\mathbf{G}\mathcal{H}}q\,dq\,U\;.
\end{equation}
Therefore, 
\begin{equation}
    \int_{\mathbf{G}\mathcal{H}}q \,dq
    \end{equation}
commutes with an irreducible representation and is thus a multiple of the identity operator. This shows that 
\begin{equation}
    \int_{\mathbf{G}\mathcal{H}}qM \,dq
    \end{equation}
is a multiple of $M.$ Similarly, the conjugation representation of the unitary group on traceless endomorphisms is irreducible, and 
\begin{equation}
    U\int_{\mathbf{G}\mathcal{H}}qMq \,dq\,U^*=U\int_{\mathbf{G}\mathcal{H}}U^*qUMU^*qU \,dq\,U^*=\int_{\mathbf{G}\mathcal{H}}qUMU^*q \,dq\;.
\end{equation}
Therefore, 
\begin{equation}
    M\mapsto \int_{\mathbf{G}\mathcal{H}}qMq \,dq
    \end{equation}
commutes with an irreducible representation and must also be a multiple of the identity operator. This completes the proof.
    \end{proof}
\begin{definition}
We have a linear map 
\begin{equation}
    \Psi:\mathcal{B}(\mathcal{H})\to \Gamma(\textup{T}^*\mathbf{G}\mathcal{H}\to\mathbf{G}\mathcal{H})\;,\;\;\Psi_M(q)=qM-qMq\;.
\end{equation}
\end{definition}
\begin{corollary}
The restriction of $\Psi$ to trace zero endomorphisms is injective.\ Furthermore, the sections in its image are $J$–holomorphic.
\end{corollary}
\begin{proof}
It's left inverse is given by
  \begin{equation}
\int:\Gamma(\textup{T}^*\mathbf{G}\mathcal{H})   \to \mathcal{B}(\mathcal{H})\;,\psi\mapsto\int_{\mathbf{G}\mathcal{H}}\psi(q)\,dq\;.
    \end{equation}   
\end{proof}
\begin{remark}
Under the metric–induced identification 
\begin{equation}
\textup{T}\mathbf{G}\mathcal{H}\to \textup{T}^*\mathbf{G}\mathcal{H}\;,\;\;A\mapsto qA\;,
\end{equation}
the vector field $q\mapsto [q,M]$ is identified with $q\mapsto q[q,M]=qM-qMq.$ It is a classical result that holomorphic vector fields on Grassmnians are identified with traceless matrices (\cite{koba}).
\end{remark}
In the following, $\mathfrak{sl}(\mathcal{H})$ denotes trace zero endomorphisms of $\mathcal{H}.$
\begin{proposition}
The map 
\begin{equation}
\mathbf{G}_k\mathcal{H}\xhookrightarrow{\Psi^{\dagger}}\mathbf{G}_{k(n-k)}\,\mathfrak{sl}(\mathcal{H})\;,\;\;q\mapsto \Psi^{\dagger}_q\,,\;\Psi^{\dagger}_q(M)=qM-qMq
\end{equation}
is well–defined and the pullback of the tautological bundle is isomorphic to $\textup{T}^*\mathbf{G}_k\mathcal{H}\to \mathbf{G}_k\mathcal{H}.$ Furthermore, the idempotent section $\mathcal{S}$ is the pullback of the idempotent section of the tautological bundle.\footnote{This is defined in \cite{Lackman1}.} For $k=1,n=2,$ this embedding is Lagrangian.
\end{proposition}
\begin{proof}
To see that the pullback of the tautological bundle is $\textup{T}^*\mathbf{G}_k\mathcal{H},$ observe that if $qA=A, Aq=0,$ then 
\begin{equation}
    qA-qAq=A
    \end{equation}
and therefore $\textup{T}^*_q\mathbf{G}_k\mathcal{H}\subset \textup{Ran}(\Psi^{\dagger}_q).$ Conversely, $A:=qM-qMq$ satisfies $qA=A, Aq=0$ and therefore $\textup{Ran}(\Psi^{\dagger}_q)\subset \textup{T}^*_q\mathbf{G}_k\mathcal{H}.$ The statement about the embedding being Lagrangian follows from dimension reasons and the fact that the embedded submanifold is isotropic, as the trace of the curvature of $\textup{T}^*\mathbf{G}_k\mathcal{H}\to \mathbf{G}_k\mathcal{H}$ is zero. 
\end{proof}
\part{The Hyperk\"{a}hler Structure of $\textup{T}^*\mathbb{P}\mathcal{H}$}
Here, we will \textit{completely} and \textit{explicitly} describe the hyperk\"{a}hler structure of $\textup{T}^*\mathbb{P}\mathcal{H}.$ Only \cref{affinee}, \cref{hols} are prerequisites.\ While the metric itself is hard to guess directly, the simplest K\"{a}hler potential turns out to work. Physically, the significance of the Hilbert–Schmidt norm, as opposed to the operator norm, is that states in quantum statistical mechanics are nonnegative operators  and expectation values are defined by pairing with the trace.
\begin{remark}
Physically, the trace class norm $q\mapsto \textup{Tr}(\sqrt{q^*q})$ is a more natural K\"{a}hler potential than the Hilbert–Schmidt norm since states have unit trace–class norm. However, the two norms agree on rank–1 projections.\ This can be seen from the fact that for a rank–one projection 
\begin{equation}
    \sqrt{q^*q}=\frac{q^*q}{\sqrt{\textup{Tr}(q^*q)}}\;.
    \end{equation}
\end{remark}
Recall that a hyperk\"{a}hler manifold can be defined as a triple $(g,I,\mathbf{J}),$ where $g$ is a Riemannian metric and $I,\mathbf{J}$ are anticommuting almost complex structures such that $gI,g\mathbf{J},gI\mathbf{J}$ are closed  2–forms (\cite{biq}).
\\\\In the following, $d^c\tau(A):=-d\tau(iA).$
\begin{lemma}
Let $\mathcal{H}$ be a complex Hilbert space.\ The norm $\tau(x):=\|x\|$ defines a K\"{a}hler potential on $\mathcal{H}\backslash\{0\},$ ie.\ $dd^c\tau$ is K\"{a}hler.\footnote{We use the convention that inner products are linear in the second argument. Also, $dd^c=i\partial\bar\partial.$}
\end{lemma}
\begin{proof}
Let $A,B\in \textup{T}_x\mathcal{H}.$ A computation shows that
\begin{align}
dd^c\tau(A,B)=\frac{i}{\|x\|}\big(\langle A,B\rangle-\langle B,A\rangle)-\frac{i}{2\|x\|^3}\big(\langle A,x\rangle\langle x,B\rangle-\langle B,x\rangle\langle x,A \rangle\big)\;.
\end{align}
It follows that
\begin{align}
dd^c\tau(iA,A)=\frac{2}{\|x\|}\|A\|^2-\frac{1}{\|x\|^3}|\langle A,x\rangle|^2\;.
\end{align}
By the Cauchy–Schwarz inequality
\begin{equation}
    \|A\|^2\ge \frac{|\langle A,x\rangle|^2}{\|x\|^2}\;,
\end{equation}
which shows that $dd^c\tau$ is K\"{a}hler.
\end{proof}
In particular, we get a K\"{a}hler potential on $\mathcal{B}(\mathcal{H})\backslash\{0\}$ by using the Hilbert–Schmidt inner product, and we can pull back the metric to $\textup{T}^*\mathbb{P}\mathcal{H}.$ This results in the following definition:
\begin{definition}
We let $g$ denote the Riemannian metric on $\textup{T}^*\mathbb{P}\mathcal{H}$ given by the real part of 
\begin{equation}
(A,B) \mapsto \frac{2}{\|q\|}\textup{Tr}(A^*B)-\frac{1}{\|q\|^3}\textup{Tr}(A^*q)\textup{Tr}(Bq^*)\;,
\end{equation}
and we let $\mathbf{J}$ be the endomorphism of $\textup{T}\textup{T}^*\mathbb{P}\mathcal{H}$ defined by $g(\mathbf{J}(A),B)=\textup{Re}(\Omega(A,B)).$
\end{definition}
\begin{proposition}
\begin{equation}
    \mathbf{J}(A)=\frac{i}{\|q\|}[q,A^*]+\frac{i}{2\|q\|^3}\textup{Tr}(A^*q)[q,q^*]\;.
    \end{equation}
\end{proposition}
\begin{proof}
Direct computation.
\end{proof}
\begin{remark}
$q\mapsto [q,q^*]$ is a natural vector field on $\textup{T}^*\mathbf{G}\mathcal{H}.$
\end{remark}
\begin{corollary}\label{hyp}
$I\mathbf{J}=-\mathbf{J}I$ and $\mathbf{J}^2=-1.$
\end{corollary}
\begin{proof}
That $I\mathbf{J}=-\mathbf{J}I$ is true by inspection.
For the second part, we compute that
\begin{equation}
 \mathbf{J}^2(A)=\frac{1}{\|q\|^2}[[q^*,A],q] +\frac{\textup{Tr}(q[q^*,A])}{2\|q\|^4}[q,q^*] +\frac{\textup{Tr}(Aq^*)}{2\|q\|^4}[[q,q^*],q]\;,
\end{equation}
where we have used the fact that $\textup{Tr}(q[q,q^*])=0.$ To see that $\mathbf{J}^2=-1,$ let $A=[q,M].$ Using the fact that $\textup{Tr}(qM)q=qMq$ for any rank–$1$ projection $q,$  we have that
\begin{align}
&\frac{1}{\|q\|^2}[[q^*,A],q]=-A+\frac{1}{\|q\|^2}\big(\Cancel[blue]{q\textup{Tr}(qq^*M)}+\Cancel[red]{\frac{qq^*}{\|q\|^2}\textup{Tr}(q^*qM)}-\Cancel[yellow]{qq^*\textup{Tr}(qM)}+\Cancel[green]{q^*q\textup{Tr}(qM)}
\\&-\Cancel[orange]{\frac{q^*q}{\|q\|^2}\textup{Tr}(qq^*M)}
-\Cancel[grey]{q\textup{Tr}(q^*qM)}\big)\;,
\end{align}
while
\begin{align}
&\frac{\textup{Tr}(q[q^*,A])}{2\|q\|^4}[q,q^*]=\frac{1}{2\|q\|^4}\big(\Cancel[yellow]{2qq^*\|q\|^2\textup{Tr}(qM)}-\Cancel[violet]{qq^*\textup{Tr}(qq^*M)}-\Cancel[red]{qq^*\textup{Tr}(q^*qM)}
\\&-\Cancel[green]{2q^*q\|q\|^2\textup{Tr}(qM)}+\Cancel[orange]{q^*q\textup{Tr}(qq^*M)}+\Cancel[pink]{q^*q\textup{Tr}(q^*qM)}\big)
    \end{align}
and
\begin{equation}
 \frac{\textup{Tr}(Aq^*)}{2\|q\|^4}[[q,q^*],q]=\frac{1}{2\|q\|^4}\Big(\textup{Tr}(q^*qM)(\Cancel[grey]{2\|q\|^2q}-\Cancel[red]{qq^*}-\Cancel[pink]{q^*q})-\textup{Tr}(qq^*M)(\Cancel[blue]{2\|q\|^2q}-\Cancel[violet]{qq^*}-\Cancel[orange]{q^*q})\Big) \;, 
\end{equation}
which proves the result.
\end{proof}
\begin{corollary}
$(g,I,\mathbf{J})$ defines a hyperk\"{a}hler structure on $\textup{T}^*\mathbb{P}\mathcal{H}.$ 
\begin{proof}
This follows from \cref{hyp}, the fact that $\Omega,\,dd^c\tau$
are closed and that $\textup{Im}(\Omega(iA,B))=\textup{Re}(\Omega(A,B)).$
\end{proof}
\end{corollary}
\begin{lemma}
For $\mathcal{H}=\mathbb{C}^2,$ $g$ is isometric to the Eguchi–Hanson metric, up to a constant.
\begin{proof}
This follows from the isomorphism with the affine quadric given in \cref{affine}, which identifies $\|q\|$ with
\begin{equation}
    \frac{1}{\sqrt{2}}\sqrt{|x|^2+|y|^2+|z|^2+1}\;.
\end{equation}
Up to a constant, this is a K\"{a}hler potential for the Eguchi–Hanson metric (\cite{ionel}, see the paragraph following equation 3.2).
\end{proof}
\end{lemma}

\end{document}